\documentclass[a4paper,12pt]{amsart}
\usepackage{amsfonts,amsthm,amsmath,amssymb,graphicx}

\title[Birkhoff spectra]{Birkhoff spectrum for piecewise monotone interval maps}

\date{}
\author{Thomas Jordan}
\address{School of Mathematics, University of Bristol, University Walk, Bristol, BS8 1TW}
\email{thomas.jordan@bristol.ac.uk}
\author{Micha\l\  Rams }
\address{ Institute of Mathematics, Polish Academy of Sciences\\
 ul. \'Sniadeckich 8, 00-656 Warszawa, Poland}
\email{rams@impan.pl}
\thanks{The second author was supported by National Science Centre grant
2014/13/B/ST1/01033 (Poland). The work was finished during the fractals and dimension programme at the Mittag-Leffler institute, both authors wish to thank the institute for their hospitality}

\theoremstyle{plain}
\newtheorem{lem}{Lemma}[section]
\newtheorem{prop}[lem]{Proposition}
\newtheorem{thm}[lem]{Theorem}
\newtheorem{cor}[lem]{Corollary}

\theoremstyle{definition}

\theoremstyle{remark}

\numberwithin{equation}{section}

\DeclareMathOperator{\diam}{diam}

\DeclareMathOperator{\var}{var}
\DeclareMathOperator{\inte}{int}

\newcommand{\R}{\mathbb R}

\newcommand{\cal}{\mathcal}
\renewcommand{\epsilon}{\varepsilon}
\begin{document}
\maketitle
\begin{abstract}
For piecewise monotone interval maps we look at Birkhoff spectra for regular potential functions. This means considering the Hausdorff dimension of the set of points for which the Birkhoff average of the potential takes a fixed value. In the uniformly hyperbolic case we obtain complete results, in the case with parabolic behaviour we are able to describe the part of the sets where the lower Lyapunov exponent is positive. In addition we give some lower bounds on the full spectrum in this case. This is an extension of work of Hofbauer on the entropy and Lyapunov spectra.
\end{abstract}

\section{Introduction}

Consider a dynamical system $(X,T)$ and a potential function $\varphi:X\to \R$. For any point $x\in X$ we can define the Birkhoff average of $\varphi$ at $x$ by

\[
A(x, \varphi) = \lim_{n\to\infty} \frac 1n \sum_{i=0}^{n-1} \varphi(T^i x)
\]

wherever the limit exists. A natural question arises: given a potential function, what is the set of points with its Birkhoff average taking a prescribed value:

\[
L_a := \{x\in X; A(x, \varphi) = a\}.
\]

Usually one describes the size of this set in terms of either topological entropy or Hausdorff dimension, the functions $a \to h_{\rm top} L_a$ and $\alpha \to \dim_H L_a$ are called entropy and dimension Birkhoff spectra of the potential $\varphi$.

This kind of problems has been studied for quite some time, possibly starting with the work of Besicovitch and Knichal, \cite{B,K} on the frequency of digits in dyadic expansions which was subsequently expanded by Eggleston and Billingsley, \cite{E},\cite{Bi}. In dynamical systems the Birkhoff spectrum was initially investigated by Pesin and Weiss (\cite{PW}, Feng-Lau-Wu \cite{FLW} and Olsen, \cite{O}.  Several other authors investigated the case where Hausdorff dimension is replaced by topological entropy including Fan, Feng and Wu, Takens and Verbitskiy and Feng and Huang \cite{FFW, TV, FH} . 

Looking at the definition one does not expect the Birkhoff spectrum to be any nicer than a measurable function. However, it turned out that for a H\"older continuous potential function on a $C^{1+\beta}$ conformal repeller the entropy Birkhoff spectrum is the Legendre transform of a suitable pressure function, thus real analytic and concave.The dimension Birkhoff spectrum is more complicated, the exception being the Lyapunov spectrum, that is the Birkhoff spectrum for the potential $\log |T'|$ ,and  the cases where the Lyapunov spectrum is trivial. In these cases the Hausdorff Birkhoff spectrum can be deduced from the entropy Birkhoff spectrum as was done by Pesin and Weiss and by Weiss, \cite{PW,W}. In the general situation, however, the dimension Birkhoff spectrum formula for conformal repellers was considered by Feng, Lau and Wu \cite{FLW} where they gave a conditional variational principe. Barreira and Saussol considered the case of H\"{o}lder continuous potential functions where they gave a formula in terms of pressure and its derivative with which they showed the spectrum is analytic but in general not concave. 

Naturally the results for one dimensional repellers were afterwards generalised for more general classes of systems, in particular ones with non-uniformly hyperbolic behaviour. We should mention here several works including, \cite{N,Oli,KS,GR, IT,GPR} for the Lyapunov spectrum and \cite{C,JJOP,CTa,IJ1,IJ2} for the Birkhoff spectrum. Of special interest for us is the paper \cite{H2} by Franz Hofbauer. In this paper he constructs the Lyapunov spectrum and entropy Birkhoff spectrum for piecewise monotone maps of the interval. 

The aim of our paper is to extend the Hofbauer's result by constructing the dimension Birkhoff spectrum for piecewise monotone maps. We are not able to completely solve the problem, we are only able to describe the intersection of the sets $L_a$ with the set $D_0$ of points with positive lower Lyapunov exponent. The formulation of our results is rather complicated and depends on the notation yet to be introduced, so we postpone it till section 3. Before that, we describe in some details the theory of piecewise monotone maps (section 2) and of the Birkhoff spectra (section 3).

\section{Piecewise monotone maps}

The theory of piecewise monotone maps was created by Hofbauer, see in particular \cite{H1} for the general theory and \cite{H2} for selected facts, relevant for multifractal formalism. In this section we will define the objects of our study and state, without proofs, the properties we are going to use.

We consider a piecewise monotonous map, $T:[0,1]\to [0,1]$ based on subintervals of the form $C_1=[0=c_1,c_2], C_2=[c_2,c_3],\ldots,C_n=[c_n,c_{n+1}=1]$. We assume that $T$ is continuous and monotonous on each open interval $(c_{i-1},c_i)$ but not necessarily continuous at the endpoints $c_i$. We will denote by $Z$ the grand orbit of the set $\{c_1,\ldots,c_{n+1}\}$, it is a countable set. 

Systems of this type can exhibit several unwelcome types of behaviour, including (but not restricted to) intervals of fixed points. To avoid those we will only consider the restriction of the full system to a fixed set $A\subset [0,1]$ with following properties:

\begin{itemize}
\item $A$ is closed,
\item $A$ is completely invariant, i.e. $A=T^{-1}(T(A)) \supset T(A)$,
\item $T|_A$ is topologically transitive,
\item the partition $(C_i)_1^n$ is generating,
\item the topological entropy $h_{\rm top}(T|_A)$ is positive.
\end{itemize}

The dynamic on $A\setminus Z$ is topologically conjugated to dynamic of the shift map on a subset of the full (one-sided) shift space with $n$ symbols. The sets $C_i$ are the first level cylinders. In general this dynamics is not of a subshift of finite type.

We call a function $f:[0,1]\to\R$ {\it regular} if at each point $x\in [0,1)$ there exists the right limit $f_+(x)=\lim_{y\to x^+} f(y)$, at each point $x\in (0,1]$ there exists the left limit $f_-(x)=\lim_{y\to x^-} f(y)$, and if for all $x\in (0,1)$ $f(x) = \frac 1 2 (f_-(x)+f_+(x))$. We will also assume that the function $\phi=\log |T'|$ is regular except possibly at $Z$.

Given $x\in [0,1]$, we denote the {\it Birkhoff average} of $f$ at $x$ by

\[
A(x,f) = \lim_{k\to\infty} \frac 1 k \sum_{i=0}^{k-1} f\circ T^i(x)
\]
if the limit exists. We denote

\[
L_a = \{x\in A; A(x,f)=a\}.
\]

We will call $A(x,\phi)$ the {\it Lyapunov exponent} and denote it by $\chi(x)=A(x,\phi)$. We will also define the upper Lyapunov exponent

\[
\overline{\chi}(x) = \limsup_{k\to\infty} \frac 1 k \sum_{i=0}^{k-1} \phi\circ T^i(x)
\]

and the lower Lyapunov exponent

\[
\underline{\chi}(x) = \liminf_{k\to\infty} \frac 1 k \sum_{i=0}^{k-1} \phi\circ T^i(x).
\]
The set of points with lower Lyapunov exponent larger than $\alpha$ will be denoted $D_\alpha$. For an invariant measure $\mu$ we will denote the Lyapunov exponent of the measure by
$$\lambda(\mu)=\int \phi d\mu.$$
It is shown by Hofbauer and Raith in \cite{HR} that for an $T$-invariant ergodic measure with positive entropy,
$$\dim_H\mu=\frac{h(\mu)}{\lambda(\mu)}.$$

We need to present in few words the definition of topological pressure we are going to use. This definition comes from \cite{H2}. We first define {\it Markov set} as a nonempty closed set $B\subset [0,1]$ for which there exists a partition $\mathcal Y$ with respect to which $T$ is piecewise monotone and such that $T(B\cap Y)\subset B$ and that either $T(B\cap Y_1)\cap Y_2=\emptyset$ or $T(B\cap Y_1) \supset B\cap Y_2$ for all $Y, Y_1, Y_2\in \cal Y$. We define by $\mathcal M(A)$ the set of all Markov subsets of $A$. We note two important properties of Markov subsets of topologically transitive positive entropy sets, see \cite{H2}.

\begin{prop} \label{supersyst}
If $A$ is topologically transitive, completely invariant and has positive entropy then ${\mathcal M}(A)$ is nonempty. Moreover, for every two Markov sets $B_1, B_2 \in {\mathcal M}(A)$ we have $B_3 \in {\cal M}(A), B_3 \supset B_1 \cup B_2$.
\end{prop}

The pressure of a potential $f$ on $A$ is then defined as

\[
P(f) = \sup_{B\in \cal M(A)} \sup_{h\in C([0,1]); h\leq f} P(T|_B,h),
\]
where $P(T|_B, h)$ is the usual topological pressure of a continuous potential on a subshift of finite type.

An important result of Hofbauer and Urba\'nski \cite{HU} shows that this definition of pressure agrees with the construction of conformal measures. Denote 
\[
C_{i_1 i_2 \ldots i_k} = C_{i_1} \cap T^{-1}(C_{i_2}) \cap \ldots \cap T^{-k+1}(C_{i_k}),
\]
those sets will be called $k$th level {\it cylinders}. We call a probabilistic measure $\mu$ supported on $A$ {\it conformal} for potential $f$ if for every $k$th level cylinder $C_{i_1 i_2 \ldots i_k}$ we have
\[
\mu(T(C_{i_1 i_2 \ldots i_k})) = \int_{C_{i_1 i_2 \ldots i_k}} e^{P(f)-f} d\mu.
\]
As the partition is generating, this formula can be iterated:

\[
\mu(T^k(C_{i_1 i_2 \ldots i_k})) = \int_{C_{i_1 i_2 \ldots i_k}} e^{kP(f)-S_k f} d\mu,
\]
where $S_k f(x) = f(x) + f(T(x))+ \ldots + f(T^{k-1}(x))$.


\begin{prop} \label{prop:conformal}
Let $A$ be topologically transitive, completely invariant and with positive entropy. Then for every regular potential there exists a conformal measure, positive on open sets.
\end{prop}

\section{Thermodynamical formalism and Birkhoff spectra}

We will present the results of Barreira and Saussol and Feng, Lau and Wu calculating the Hausdorff dimension Birkhoff spectrum for interval expanding maps  and then we will proceed to Johansson, Jordan, Pollicott and \"Oberg's generalisation to the nonuniformly hyperbolic situation.

Let $T$ be a uniformly hyperbolic $C^{1+\beta}$ map on the interval $I$. Let $f:I\to \R$ be a continuous potential. We denote

\[
P(q,a,\delta) :=P(q(f-a)-\delta \phi)
\]

For $a\in \text{int} H$ $P(-\infty, \delta)=P(\infty, \delta)=\infty$ for all
$\delta$. Hence, there exists $P_{\rm min}(\delta) = \inf_q P(q, a,\delta)$. In the case when $f$ is H\"{o}lder continuous these pressure functions will be analytic. Moreover in the uniformly hyperbolic situation  as $\phi$ is bounded away from zero, so is $\partial P(q,a,\delta)/\partial\delta$. Hence, one can find $\delta_0=\delta_0(a)$ for which $P_{\rm min}(\delta_0)=P(q_0,a,\delta_0)=0$. The main result of \cite{BS} and \cite{FLW} is the following

\begin{prop}
For $a\in H$
\begin{equation}\label{cvp}
\dim_H L_a =\max\left\{\frac{h(\mu)}{\lambda(\mu)}:\int f d\mu =a\text{ and $\mu$ $T$-invariant}\right\}.
\end{equation}
If additionally $f$ is H\"{o}lder continuous and $a\in\rm{int} (H)$ then
\begin{equation}\label{delta}
\max \left\{\frac{h(\mu)}{\lambda(\mu)}:\int f d\mu =a\text{ and $\mu$ $T$-invariant}\right\}=\delta_0(a)
\end{equation}
\end{prop}
\begin{proof}
We will outline the details of the proof.  To show \eqref{cvp} ,as done in \cite{FLW} , the idea is to adapt the ideas of the proof the classical variational principle.

Note first that at $q_0$ the function $P(\cdot, a, \delta_0)$ has a minimum, hence $\partial P(q_0,a,\delta_0)/\partial q_0 =0$. At the same time,

\[
\frac {\partial P(q_0,a,\delta_0)} {\partial q_0} = \int (f-a) d\mu_{q_0, \delta_0},
\]

where $\mu_{q_0, \delta_0}$ is the unique ergodic equilibrium state for the potential $\psi = q_0(f-a) -\delta \phi$. Thus $\int f d\mu_{q_0,\delta_0}=a$ and 
\[
h(\mu_{q_0,\delta_0})+\int (f-a) d\mu_{q_0, \delta_0}-\delta\int\phi d\mu_{q_0,\delta_0}=0.
\]
Rearranging this gives that
$$\frac{h(\mu_{q_0,\delta_0})}{\int\phi d\mu_{q_0,\delta_0}}=\delta_0.$$
On the other hand for an arbitrary $T$-invariant probability measure $\mu$ with $\int f d\mu=a$ it follows by the variational principle that
\[
h(\mu_{q_0,\delta_0})+\int (f-a) d\mu_{q_0, \delta_0}-\delta\int\phi d\mu_{q_0,\delta_0}\leq 0.
\]
and so $\frac{h(\mu)}{\int\phi d\mu_{q}}\leq \delta_0$. Thus 
$$\max\left\{\frac{h(\mu)}{\lambda(\mu)}:\int f d\mu =a\text{ and $\mu$ $T$-invariant}\right\}=\delta_0(a).$$
\end{proof}

In the case where $f$ is continuous but not H\"{o}lder continuous, the pressure is not always differentiable so the definition of $\delta_0(a)$ given above may not work. However if instead we define
\begin{equation}\label{deltadef}
\delta_0(a)=\sup\{\delta:\inf\{P(q(f-a)-\delta\phi):q\in\R\}>0\}
\end{equation}
then it is possible to show that for $a\in H$,
$$\delta_0(a)=\max\left\{\frac{h(\mu)}{\lambda(\mu)}:\int f d\mu =a\text{ and $\mu$ $T$-invariant}\right\}.$$
As we will be working in the case where we certainly cannot assume the pressure is differentiable this is the definition we will use.

In the nonuniformly hyperbolic situation the picture is more complicated. Let now $T$ be $C^{1+\beta}$ expansive Markov map, but we allow the existence of parabolic  fixed points $p_1,\ldots, p_\ell$). We take a continuous function $f:[0,1]\to [0,1]$ as before and let $H_p$ be the convex hull of $\{f(p_1),\ldots,f(p_{\ell})\}$. In this case it is shown in \cite{JJOP} that if $a\in H_p$ then $\dim_H L_a=1$ and if $a\in (H\backslash H_p)$ then as before
$$\dim_H L_a=\max\left\{\frac{h(\mu)}{\lambda(\mu)}:\int f d\mu =a\text{ and $\mu$ $T$-invariant}\right\}=\delta_0(a)$$
where $\delta_0(a)$ is defined as in \eqref{deltadef}. However in $H_p$ it will not always be the case that $\delta_0(a)=1$. We will show that $\delta_0(a)$ is lower-semi continuous where as in \cite{IJ2} an example is given where $a\to \dim_H L_a$ is not lower-semi continuous. It turns out that $\delta_0(a)$ will give the hyperbolic dimension of $L_a$ (points with positive local Lyapunov exponent) but not always the   full Hausdorff dimension.

\subsection{The main result}

We now turn back to our more general setting.
Denote by $H$ the set of all $a\in\R$ for which there exists a sequence of non-atomic ergodic measures $\mu_n$, each supported on a Markov subset of $A$, such that $\int f d\mu_n \to a$. This is equivalent to the definition of $H$ given in \cite{H2}. By Proposition \ref{supersyst}, $H$ is a closed interval. While it is not known whether $A(x,f)\in H$ for all $x\in A$, Hofbauer proves in \cite{H2} that for all $a\notin H$ $h_{\rm top}L_a = 0$, and hence $\dim_H L_a\cap D_0=0$. In what follows we will restrict our attention to $a\in H$.

The set $H$ can be divided into two subsets. Let $H_p$ be the set of all $a\in H$ such that we can not only choose a sequence $\mu_n$ as in the definition of $H$, but in addition the measures $\mu_n$ can be chosen as to satisfy $\int \phi d\mu_n \to 0$. Then $H_h = H \setminus H_p$. The sets $H_p$ and $H_h$ will be called {\it parabolic} and {\it hyperbolic}.

By Lemma \ref{endpoints}, which will be proved in the next section, for $a\in\text{int}(H)$ it makes sense to define
$$\delta_0(a)=\sup\{\delta:\inf\{P(q(f-a)-\delta\phi):q\in\R\}>0\}.$$
We will also define
$$\dim_{\rm{hyp}}=\sup\{\dim M:\text{ $M$ is Markov and }M\subseteq A\}.$$
It can be also shown that this is both the supremum of dimensions of hyperbolic ergodic measures supported on Markov subsets of $A$ and the supremum of $\delta_0(a)$ where $a\in H$.
\begin{thm} \label{thm:main}
\begin{enumerate}
\item
If $a\in H$ then $\dim_H(L_a\cap D_0)=\delta_0(a)$. 
\item
For all $a\in \text{int}(H)$
$$\delta_0(a)=\sup\left\{\frac{h(\mu)}{\lambda(\mu)}:\mu\text{ $T$-ergodic and }\int f d\mu=a\right\}$$
\item
If $a\in H_p$ then $\dim_H(L_a)\geq \dim_{\text{hyp}}(A)$. Moreover if $a\in \rm{int}(H_p)$ then $\dim_H(L_a\cap D_0)=\dim_{\text{hyp}}(A)$.
\item
The function $a\to\dim_H L_a\cap D_0$ is continuous on $H_h$ and lower semi-continuous on $H$.
\end{enumerate}
\end{thm}

Note that we are using the convention that $\dim_H\emptyset=-\infty$ and that if an endpoint $a$ of $H$ is also in $H_p$ then it is possible that $\delta_0(a)=-\infty$ (in this case $L_a$ will be non-empty and have dimension at least the hyperbolic dimension but $L_a\cap D_0=\emptyset$).

The rest of the paper will be structured as follows. In section 4 we will proof the result on the dimension of $L_{a}\cap D_0$, in section 5 we will  prove the results for end points and for $H_p$ and in section 6 we will proof the continuity properties of the spectrum. Finally in section 7 we will discuss some examples.

\section{Dimension of $L_a\cap D_0$} 
\subsection{Preparation -- exceptional sets}

We would like to work with symbolic dynamics and then apply the results to the geometry of some invariant sets. The symbolic dynamics is not that complicated, we are working with closed invariant positive entropy subsets of a full shift. But the interplay between dynamics and geometry is rather involved.

We will work with more general partitions than the original partition $(C_i)$, assuming however that $T$ is piecewise monotone for them (hence, all the results of Section 2 hold for them), for example every subpartition of $(C_i)$ has this property.
Let $\cal Z$ be such a partition. Denote by $Z_0$ the partition points of $\cal Z$ and let $Z_k=T^{-1}(Z_{k-1})$. The elements of partition ${\cal Z}^k$ generated by points $Z_0\cup\ldots\cup Z_{k-1}$ will be called $k$-level cylinders.

If both endpoints of a $k$ level cylinder $I$ belong to $Z_1\cup\ldots\cup Z_{k-1}$ then $T(I)$ is a $k-1$ level cylinder. However, if one or both of the endpoints of $I$ belongs to $Z_0$ then $T(I)$ is in general only a subset of some $k-1$ level cylinder. We would like to get some bounds on how small can $|T^k(I)|$ be.

For the upper bound we will work with full system $(A,T)$ using part of the statement of \cite[Lemma 14, 15]{H2}. Given $x\in [0,1]$, let $Z_k(x)$ be the $k$ level cylinder containing $x$. Let $\mu$ be any probabilistic measure on $[0,1]$, positive on all cylinders (of any level) intersecting $A$. Given $d>0$ let

\[
N_d = \bigcup_{k=0}^\infty Z_k \cup \{x\in A; \limsup_{k\to\infty} \mu(T^k(Z_k(x))) < d\}.
\]
Denote also 

\[
\var_{\cal Z} f = \sup_{Z\in\cal Z} \sup_{x_1,x_2\in Z} |f(x_1)-f(x_2)|.
\]

We have 

\begin{lem} \label{lem:nd}
For every $\alpha>\var_{\cal Z} \phi$ we have
\[
\lim_{d\to 0} \dim_H (N_d\cap D_\alpha) =0
\]
\end{lem}

For the lower bound we need to work with Markov subsets. Symbolically they are just subshifts of finite type. We will need the following lemma.

\begin{lem} \label{lem:mark}
Let $(B, \cal Z)$ be a Markov set. Then there exists $K>0$ such that for every $k$ for every $k$th level cylinder $I$ with nonempty intersection with $B$ we have

\[
|T^k I| >K.
\]
Assume in addition that $T|B$ is topologically transitive. Then for every $x\in B$ there exist $r>0$ and $l>0$ such that for every $k$ for every $k$th level cylinder $I$ there exists interval $J\subset T^k(I)$ and an integer $j\leq l$ such that $T^j(J)=B(x,r)$ and $T^j$ is continuous and monotone on $J$.
\begin{proof}
Consider any $k$th level cylinder $I$ contained in a first level cylinder $J_0$ and intersecting $B$. We have $T(I)\subset I' \subset J_1$, where $I'$ is a $k-1$ level cylinder and $J_1$ is a first level cylinder. We have $I_1 \cap B \supset T(I)\cap B \neq\emptyset$, hence $J_1\cap B = T(J_0\cap B)$ by the definition of Markov set. Hence, the set $I'\cap B = B\cap T(I)$. Following the same reasoning, $T^k(I)$ is contained in some first level cylinder $J_2$ and $J_2\cap B = B\cap T^k(I)$. Then we can set $K$ as the minimum (over all first level cylinders $J$) of $|{\rm span} (J\cap B)|$ and the first part of the assertion is proved.

The second part follows easily because every open interval of a given size centred at a point in $A$ contains a (symbolic) cylinder of a bounded level.
\end{proof}
\end{lem}

\subsection{Upper bound} \label{sec:up}
For convenience from now on we will denote
$$P(a,q,\delta)=P(q(f-a)-\delta\phi)$$
We will obtain the upper bound in the dimension formula with the use of conformal measures. 
Let $a$ and $\delta$ be such that $\inf_q P(a, q, \delta)\leq 0$.
Fix $\alpha > 0$ and $\varepsilon >0$. Consider ${\cal Y}$, a subdivision of ${\cal Z}$ such that
$\var_{\cal Y} \phi < \alpha$. As for arbitrarily small $ \alpha$ we can always find such ${\cal Y}$, we need only to prove that

\[
\dim_H ((L_a \setminus N_d(\mu, {\cal Y}))\cap D_\alpha) \leq \delta_0 + \gamma
\]
for arbitrarily small $\gamma$. Indeed, as the Hausdorff dimension of $N_d \cap D_\alpha$ will go to 0 together with $d$, we will get an upper bound for the Hausdorff dimension of $L_a \cap D_\alpha$, we will then pass with $\alpha$ to 0 to obtain the assertion.

Consider $q_0$ such that $P(a,q_0, \delta_0)=p \in [0,\varepsilon)$. Let $\mu$ be the $q_0(f-a) -
\delta_0 \phi -p$-conformal measure. Choose $d>0$. Denote by $G_d(n)$ the set of cylinders $I\in {\cal Y}^n$ for which $\mu(T^n(I))>d$. Choose a small $\epsilon$ and denote

\[
X_{a,n,\epsilon} = \{I\in G_d(n); a-\epsilon < A_N(x,f) < a+\epsilon,\ A_n(x,\phi) > \alpha - \epsilon\ \forall
x\in I\}.
\]
Note that we have
\[
|I| \leq e^{n (\alpha - \epsilon - \var_{\cal Y}\phi)}
\]
for all $I\in X_{a,n,\epsilon}$. Fixing $n\in N$, we have

\[
1= \sum_{I\in {\cal Y}^n} \mu(I) \geq \sum_{I\in G_d(n)} \mu(I) \geq \sum_{I\in
X_{a,n,\epsilon}} \mu(I).
\]

Substituting the definition of $\mu$, we get

\[
\mu(I) \geq \inf_{x\in I} \exp(-q_0 S_n (f-a) +np) |(T^n)'|^{\delta_0}(x)
\cdot \mu(T^n(I)) \geq e^{-n\epsilon |q_0|+np} e^{-n \delta_0 \var_{\cal Y}} |I|^{\delta_0} d.
\]

Hence, for any $\gamma$ such that

\[
\gamma > \frac {\epsilon (|q_0|+1) +
\delta_0 \var_{\cal Y}\phi} {\alpha - \epsilon - \var_{\cal Y}\phi}
\]

we have

\[
\sum_n \sum_{I\in X_{a,n,\epsilon}} |I|^{\delta_0 + \gamma} \leq d^{-1} \sum_n
e^{-n(\epsilon-p)} < \infty,
\]

which implies that

\[
\dim_H ((L_a \setminus N_d(\mu, {\cal Y}))\cap D_\alpha) \leq \delta_0 + \gamma.
\]
As $\epsilon$ and $\var_{\cal Y}$ can be chosen arbitrarily small,
$\gamma$ also can be arbitrarily small. We are done.

\subsection{Lower bound}

We start with the following lemma
\begin{lem}\label{endpoints}
For all $\delta\in\R$ and $a\in\text{int}(H)$,
$$\lim_{q\to-\infty}P(q(f-a)-\delta\phi)=\lim_{q\to\infty}P(q(f-a)-\delta\phi)=\infty.$$
\end{lem}
\begin{proof}
Since $a\in\inte H$ there exist $a_1<a<a_2$ such that $a_1,a_2\in H$. Thus there exist hyperbolic measures $\mu_1,\mu_2$ such that $\int f\text{d}\mu_1=a_1$ and $\int f\text{d}\mu_2=a_2$. Thus
$$P(q(f-a)-\delta\phi)\geq q(a_1-a)+h(\mu)-\delta\int\phi\text{d}\mu$$
and since $\phi$ is bounded we have that $\lim_{-q\to\infty}P(q(f-a)-\delta\phi)=\infty.$ We also have that
$$P(q(f-a)-\delta\phi)\geq q(a_2-a)+h(\mu)-\delta\int\phi\text{d}\mu$$
and so $\lim_{q\to\infty}P(q(f-a)-\delta\phi)=\infty$.
\end{proof}

We can now proceed with the lower bound for Theorem \ref{thm:main}, starting from the interior of the spectrum.
We will fix $a$ in the interior of $H$ and we will choose $\delta$ such that
$$\gamma:=\inf_{q\in\R}\{P(q(f-a)-\delta\phi)\}>0.$$
We show that
$$\dim_H L_a\cap D_0\geq\delta$$
via the following key lemma. 
\begin{lem} \label{lem:lowerint}
If $a\in\text{int}(H)$ and $\delta>0$ is such that
$$\gamma:=\inf_{q\in\R}\{P(q(f-a)-\delta\phi)\}>0$$
then there exists an ergodic measure $\mu$ such that $\int f\rm{d}\mu=a$, $\lambda(\mu) > \gamma/2$, and $\frac{h(\mu)}{\lambda(\mu)}\geq\delta$.
\end{lem}
\begin{proof}
We start by finding $c$ and an interval $I=[x_0,x_1]$ which contains $c$ such that $P(c(f-a)-\delta\phi)=\gamma$, where $P(x_0(f-a)-\delta\phi)>2\gamma$
and $P(x_1(f-a)-\delta\phi)>2\gamma$. We can do this by Lemma \ref{endpoints}. Since $f$ and $\phi$ are regular we can find step functions $\tilde{f}$ and $\tilde{\phi}$ such that
\begin{enumerate}
\item
$|f-\tilde{f}|<\frac{\gamma}{2(x_1-x_0)}$,
\item
For all $q\in I$,
$$|q(f-a)-\delta\phi-(q(\tilde{f}-a)-\delta\tilde{\phi})|\leq\frac{\gamma}{4}.$$
\end{enumerate}
Thus we will have that by Lemma 6 of \cite{H2}
\begin{eqnarray*}
P(c(\tilde{f}-a)-\delta\tilde{\phi})&<&5\gamma/4)\\
P(x_0(\tilde{f}-a)-\delta\tilde{\phi})&>&7\gamma/4\\
P(x_1(\tilde{f}-a)-\delta\tilde{\phi})&>&7\gamma/4\\
\forall q\in I, P(q(\tilde{f}-a)-\delta\tilde{\phi})&>&3\gamma/4).
\end{eqnarray*}
Thus we can find a Markov set $B$ such that $\tilde{f}$ and $\tilde{\phi}$ are constant on all elements of the partition for $B$ and where
\begin{eqnarray*}
P_{B}(c(\tilde{f}-a)-\delta\tilde{\phi})&<&5\gamma/4)\\
P_{B}(x_0(\tilde{f}-a)-\delta\tilde{\phi})&>&7\gamma/4\\
P_{B}(x_1(\tilde{f}-a)-\delta\tilde{\phi})&>&7\gamma/4\\
\forall q\in I, P_B(q(\tilde{f}-a)-\delta\tilde{\phi})&>&3\gamma/4).
\end{eqnarray*}
We have that the function $q(\tilde{f}-a)-\delta\tilde{\phi})$ when restricted to $B$ will have unique ergodic, non-atomic equilibrium states for all $q\in\R$. Moreover $q\to P_B(q(\tilde{f}-a)-\delta\tilde{\phi})$ is analytic and the derivative with at the point $q$ will be given by $\int \tilde{f}\rm{d}\mu_q-a$. We also know that the measures $\mu_q$ will vary weak-*-continuously with $q$. By the mean value theorem we can see that
$$\{\int \tilde{f}\rm{d}\mu_q:q\in I\}\supseteq \left[a-\frac{\gamma}{2(x_1-x_0)},a+\frac{\gamma}{2(x_1-x_0)}\right].$$
Thus we know that there exist $q_0,q_1\in I$ such that $\int f\rm{d}\mu_{q_0}<a$ and $\int f\rm{d}\mu_{q_1}>a$. Since $f$ is discontinuous at most countably many points by weak-*-continuity the function $q\to\int f\rm{d}\mu_q$ will vary continuously. Thus there exists $q_2\in I$ such that $\int f\rm{d}\mu_{q_2}=a$. 

We have, since $\mu_{q_2}$ is an equilibrium state,
$$h(\mu_{q_2})+\int(q_2(\tilde{f}-a)-\delta\tilde{\phi})\rm{d}\mu_{q_2}=P_{B}(q_2(\tilde{f}-a)-\delta\tilde{\phi})>3\gamma/4$$
Thus since
$$|q_2(f-a)-\delta\phi-(q_2(\tilde{f}-a)-\delta\tilde{\phi})|<\frac{\gamma}{4}$$
we know that
$$h(\mu_{q_2})+\int(q_2(f-a)-\delta\phi)\rm{d}\mu_{q_2}>\gamma/2.$$
Since $\int f\rm{d}\mu_{q_2}=a$ we can conclude that $h(\mu_{q_2})-\delta\int\phi\rm{d}\mu_2>0$ and since $0<h(\mu_{q_2})\leq\lambda(\mu_{q_2})$ we have that $\lambda(\mu_{q_2}) > \gamma/2$. Hence,
$$\frac{h(\mu_{q_2})}{\lambda(\mu_{q_2})}\geq\delta.$$
\end{proof}
By combining the Birkhoff ergodic theorem and the fact proved in \cite{HR} that in this setting the dimension of an ergodic measure with positive entropy is given by the entropy divided by the Lyapunov exponent we can conclude that $\dim_H L_a\cap D_0\geq\delta_{0}(a)$ for any $a\in \text{int}(H)$. Combining this with our upper bound shows that for $a\in \text{int}(H)$, $\dim_H L_a\cap D_0=\delta_{0}(a)$. So for part one of  Theorem \ref{thm:main} the remaining is the lower bound for the endpoints which is done in the next section. Part 2 of the Theorem is an immediate corollary of what we have just done 
\begin{cor}
For $a\in\text{int}(H)$ we have that
$$\delta_0(a)=\sup\left\{\frac{h(\mu)}{\lambda(\mu)}:\mu\text{ $T$ ergodic and }\int f d\mu=a\right\}.$$
\end{cor}
\begin{proof}
This is proved by combining Lemma \ref{lem:lowerint} and the fact that the dimension of any ergodic measure with positive entropy which has integral $a$ will give a lower bound for $\dim_H L_a\cap D_0$.
\end{proof}


\section{The lower bounds for the parabolic part and the end points}

\subsection*{Limiting measures}
The lower bound at the endpoints of the spectrum will be obtained with the use of the following theorem which will also be used for our lower bound for the dimension of $L_a$ when $a\in H_p$

\begin{thm} \label{thm:w}
Let $\{(B_i, {\cal Y}_i)\}$ be a sequence of Markov sets contained in the topologically transitive completely invariant set $A\subset [0,1]$. For each $i$ let $\mu_i$ be a Gibbs measure supported on $B_i$ with Hausdorff dimension $d_i$, Lyapunov exponent $\lambda_i>0$, entropy $h_i$, and let $a_i = \int f d\mu_i$. Assume that $a_i\to a$. Then there exist a measure $\mu$ supported on $A$ such that

\[
\dim_H \mu \geq \limsup d_i
\]
and for a $\mu$-typical point $x$ we have $A(x,f)=a$ and $\underline{\chi}(x)=\liminf \lambda_i$.
\begin{proof}
The proof will use the $w$-measures technique and will be a slightly modified version of \cite[Proposition 9]{GR}.

By passing to a subsequence we might assume that the sequence $\{d_i\}$ converges to a limit, let us denote it by $d$.
By Proposition \ref{supersyst} we can freely assume that $T|B_i$ is topologically transitive and $B_i\cap B_{i+1}\neq\emptyset$ for all $i$. We choose a fast decreasing sequence $\varepsilon_i\searrow 0$ such that $\varepsilon_i $ is sufficiently small compared with $\lambda_i$ and $h_i$.

Our first step is preparation of our Markov sets. The dynamics $T|_{B_i}$ is a subshift of finite type, we can choose a subpartition ${\cal Z}_i$ such that $T|_{B_i}$ is a subshift of finite type in the new partition as well and that $\var_{{\cal Z}_i} f, \var_{{\cal Z}_i} \phi < \varepsilon_i$. Let $x_i\in B_i\cap B_{i+1}$ and let $K_i, r_i, l_i$ be as given by Lemma \ref{lem:mark} as applied to $(B_i, {\cal Z}_i, x_i)$. We assume that the partition ${\cal Z}_{i+1}$ is so detailed with respect to ${\cal Z}_i$ that the ball $B(x_i, r_i)$ contains a first level cylinder $J_{i+1}\in {\cal Z}_{i+1}$. Denote $L=\sup |T'|$.

We will construct inductively a fast increasing sequence $\{m_i\}$ and will, also inductively, construct $\mu$. Let us begin the induction by considering the cylinders of level $m_1+1$ for $(T, B_1, {\cal Z}_1)$. Denote by ${\cal I}_1$ the family of those cylinders $I\in {\cal Z}_1^{m_1+1}$ for which for all $x\in I\cap B_1$ we have

\begin{equation} \label{aa}
|\sum_{k=0}^{m_1-1} f\circ T^k(x) - m_1 a_1| < 2 m_1 \varepsilon_1,
\end{equation}

\begin{equation} \label{ab}
|\sum_{k=0}^{m_1-1} \phi\circ T^k(x) - m_1 \lambda_1| < 2 m_1 \varepsilon_1,
\end{equation}
and

\[
|\log \mu_1(I) + m_1 h_1| < 2 m_1 \varepsilon_1.
\]

By Birkhoff Theorem and Shannon-McMillan-Breiman Theorem, for $m_1$ big enough we have

\[
\sum_{I\in {\cal C}_1} \mu_1(I) > 1-\varepsilon_1.
\]

By our assumption about $\var_{{\cal Z}_1} f, \var_{{\cal Z}_1} \phi$, \eqref{aa} and \eqref{ab} imply

\[
|\sum_{k=0}^{m_1-1} f\circ T^k(x) - m_1 a_1| < 3 m_1 \varepsilon_1,
\]
and

\[
|\sum_{k=0}^{m_1-1} \phi\circ T^k(x) - m_1 \lambda_1| < 3 m_1 \varepsilon_1,
\]

for all $x\in I$. For each of those intervals Lemma \ref{lem:mark} gives us some $j$ which can take only $l_1$ different values. We have thus for some $j_1\leq l_1$ at least $\frac 1 {l_1} \exp(m_1(h_1-3\varepsilon_1))$ intervals $I$ of length at least $L^{-l_1} \exp(m_1(-\lambda_1-3\varepsilon_1)) |J_2| > \exp(m_1(-\lambda_1-4\varepsilon_1))$ such that $T^{m_1+j_1}(I)=J_2$ and that for all $x\in I$

\[
|\sum_{k=0}^{m_1+j_1-1} f\circ T^k(x) - m_1 a_1| < 3 m_1 \varepsilon_1 + j_1 \sup|f| < 4 (m_1+j_1) \varepsilon_1,
\]
and

\[
|\sum_{k=0}^{m_1+j_1-1} \phi\circ T^k(x) - m_1 \lambda_1| < 3 m_1 \varepsilon_1 + j_1 \sup |\phi| < 4 (m_1+j_1) \varepsilon_1.
\]
We equidistribute the measure $\mu$ over all $I\in {\cal C}_1$.

Now we are ready for the inductive step. For $m_i$ we have at least $\exp(m_i(h_i-4\varepsilon_i))$ intervals $I$ with following properties:

\[
\exp(m_i (-\lambda_i - 4 \varepsilon_i)) \leq \diam I,
\]

\[
m_i (a_i - 4 \varepsilon_i) \leq m_i A_{m_i+j_i}(x,f) \leq m_i (a_i + 4 \varepsilon_i),
\]

\[
\exp(m_i (-h_i - 3\varepsilon_i)) \leq \mu(I) \leq \exp(m_i (-h_i + 3 \varepsilon_i)),
\]

\[
T^{m_i+j_i}(I)=J_{i+1}.
\]

Applying Birkhoff ergodic theorem, Shannon-McMillan-Breiman theorem and the law of iterated logarithm we can find in $J_{i+1}$ a family ${\cal D}_{i+1}$ of $(m_{i+1}-m_i-j_i+1)$ level cylinders for $(T, B_{i+1}, {\cal Z}_{i+1})$ for which the following are true:

\begin{itemize}
\item[i)] For any $k=1,\ldots,m_{i+1}-m_i-j_i$ for any $x\in I\cap B_{i+1}, I \in {\cal D}_{i+1}$
\[
|A_k(x,f) -ka_{i+1}| \leq m_i \varepsilon_i + 2 k \varepsilon_{i+1},
\]
\item[ii)] For any $k=1,\ldots,m_{i+1}-m_i-j_i$ for any $x\in I\cap B_{i+1}, I \in {\cal D}_{i+1}$
\[
|A_k(x,\phi) -k\lambda_{i+1}| \leq m_i \varepsilon_i + 2 k \varepsilon_{i+1},
\]
\item[iii)] For any $k=1,\ldots,m_{i+1}-m_i-j_i$ for any $x\in I\cap B_{i+1}, I\in {\cal D}_{i+1}$, the $k$ level cylinder $J$ of $(T, B_{i+1}, {\cal Z}_{i+1})$ containing $x$ satisfies
\[
|\log \mu_{i+1}(J) + k h_{i+1}| \leq m_i \varepsilon_i + 2 k \varepsilon_{i+1}.
\]
\end{itemize}

If $m_{i+1}$ is sufficiently big (depending on $m_i$, $\varepsilon_i$, $\varepsilon_{i+1}$) then

\[
\mu_{i+1} \bigcup_{I\in {\cal D}_{i+1}} I \geq (1-\varepsilon_{i+1}) \mu_{i+1}(J_{i+1}).
\]

By our assumption about $\var_{{\cal Z}_1} f, \var_{{\cal Z}_1} \phi$, points i) and ii) imply that for all $x\in I, I\in {\cal D}_{i+1}, k=1,\ldots,m_{i+1}-m_i-j_i$

\begin{equation} \label{ba}
|A_k(x,f) -ka_{i+1}| \leq m_i \varepsilon_i + 3 k \varepsilon_{i+1},
\end{equation}
and

\begin{equation} \label{bb}
|A_k(x,\phi) -k\lambda_{i+1}| \leq m_i \varepsilon_i + 3 k \varepsilon_{i+1},
\end{equation}

Like in the first step, we then find for each of those cylinders the time $j_{i+1}$ given by Lemma \ref{lem:mark} and we choose some of them with the same $j_{i+1}$, at least $1/l_{i+1}$ of the total number. In the chosen cylinders we find subintervals which are moved by $T^{m_{i+1}-m_i-j_i+j_{i+1}}$ onto $J_{i+2}$. 

We denote by ${\cal E}_{i+1}$ the set of those chosen subintervals $J$ and let 
\[
{\cal C}_{i+1} = \{T^{-m_i-j_i}_I(J); J\in {\cal E}_{i+1}, I \in {\cal C}_i\}.
\]
We equidistribute $\mu$ on ${\cal C}_{i+1}$.

Note that, provided $m_{i+1}$ is sufficiently big,

\[
m_{i+1} \varepsilon_{i+1} > 4 m_i \varepsilon_i + j_{i+1} \sup (|f|+|\phi|)
\]
and hence \eqref{ba}, \eqref{bb} and iii) give us back the inductive assumption for the step $i+1$.

This way, we have constructed a measure $\mu$ with following properties:

\begin{itemize}
\item[a)] For any $l=m_i+k<m_{i+1}$ for $\mu$-almost every $x$ we have
\[
|l\cdot A_l(x,f) - m_i a_i - k a_{i+1}| \leq 4m_i \varepsilon_i + 4 k \varepsilon_{i+1},
\]

\item[b)] For any $l=m_i+k<m_{i+1}$ for $\mu$-almost every $x$ we have
\[
|l\cdot A_l(x,\phi) - m_i \lambda_i - k \lambda_{i+1}| \leq 4m_i \varepsilon_i + 4 k \varepsilon_{i+1},
\]

\item[c)]
For any $l=m_i+k<m_{i+1}$, for any interval $I$ of length $\exp(m_i(-\lambda_i - 4\varepsilon_i) + k(-\lambda_{i+1} - 4 \varepsilon_{i+1}))$ we have
\[
\mu(I) \leq 2 (1-\varepsilon_{i+1})^{-2} \exp(m_i(-h_i+4 \varepsilon_i) + k(-h_{i+1}+4 \varepsilon_{i+1})).
\]
\end{itemize}

The points a), b) follow from \eqref{ba}, \eqref{bb}. The point c) follows from iii) because, by \eqref{bb}, the interval of this length can intersect at most two of the sets $T_C^{-m_i}(G)$ of positive measure $\mu$, where $G$ is a $k$ level cylinder of $(T, B_{i+1}, \cal Z_{i+1})$. Point a) implies a $\mu$-typical point has Birkhoff average of $f$ equal to $a$. Point b) implies, by Frostman's lemma, that the Hausdorff dimension of $\mu$ is at least $d$.
\end{proof}
\end{thm}
\subsection{Endpoints of $H$}
Let us now take $a=\max H$ (the case $a=\min H$ is analogous), assume that $\delta_0(a)\neq-\infty$ and thus we can assume that $P(a, q, \delta)\geq \gamma >0$ for some $\delta$ and all $q\in \R$. As $a=\max H$, $P(a, q, \delta)$ is nonincreasing and

\[
\lim_{q\to\infty} \frac d {dq} P(a, q, \delta) = 0.
\]

Moreover, as $\lim_{q\to\infty} P(a, q, \delta) > -\infty$, we have 

\[
\limsup_{q\to\infty} q \frac d {dq} P(a, q, \delta) = 0.
\]

Let us choose some sequence $q_i\to\infty$ with $q_i \frac d {dq} P(a, q, \delta)_{|q=q_i}>-\gamma/2$. Let 

\[
a_i = a + \frac d {dq} P(a, q, \delta)_{|q=q_i},
\]

assume that $a_i\in H$. Then, as $P(a', q, \delta) = q(a'-a) + P(a, q, \delta)$, we have 

\begin{eqnarray*}
P(a_i, q_i, \delta) &>& \frac 12 \gamma,\\
\frac d {dq} P(a_i, q, \delta)_{|q=q_i} &=& 0.
\end{eqnarray*}

By convexity of $P$, $q_i$ is the minimum point of $P(a_i, \cdot, \delta)$. We can thus apply Lemma \ref{lem:lowerint} to construct an ergodic measure $\mu_i$ such that

\begin{eqnarray*}
\lambda(\mu_i) &>& \gamma_2,\\
\int f d\mu_i &=& a_i,\\
\frac {h(\mu_i)} {\lambda(\mu_i)} &\geq& \delta.
\end{eqnarray*}

Applying Theorem \ref{thm:w} we can construct a set of points with Birkhoff average $a$ and this set has Hausdorff dimension $\delta$. Moreover, as all the measures $\mu_i$ have Lyapunov exponents uniformly bounded away from zero, all those points are in $D_{\gamma/2}$. Hence,

\[
\dim_H L_a \cap D_0 \geq \delta.
\]

\subsection{The full spectrum in $H_p$}


We use the following preliminary result
\begin{lem}\label{invapprox}
Let $\mu$ be a $T$-inv prob measure with support contained on a Markov set $M$. We can find a sequence of measures $\mu_n$  each supported on $M$ such that each $\mu_n$ is ergodic and non-atomic, $\mu$ is a weak* limit of $\{\mu_n\}$ and $\lim_{n\to\infty}h(\mu_n)=h(\mu)$.
\end{lem}
\begin{proof}
Let $\eta$ be the measure of maximal entropy (Parry measure for M). We let $\nu_n=(1-n^{-1})\mu_n+n^{-1}\eta$. Note that $\nu_n$ is $T$-inv and fully supported on $M$. Thus we can apply Lemma 4.2 from \cite{O} to find a sequence of Gibbs measures (so non-atomic and ergodic) converging to $\nu_n$ weak* and in terms of entropy. The result now easily follows. 
\end{proof}
We can now deduce that $H_p$ is an interval

\begin{lem} \label{lem:inter}
$H_p$ is a (possibly empty or trivial) closed interval.
\end{lem}
\begin{proof}
The fact that $H_p$ is a closed set is immediate from the definition. To see it is an interval we will show that if $a_1,a_2\in H_p$ and $p\in (0,1)$ then $pa_1+(1-p)a_2\in H_p$.

To see this note that if $\mu$ and $\nu$ are both supported on subsets of Markov sets then we can find a Markov set $M$ such that $p\mu+(1-p)\nu$ support is contained on $M$. We can then apply Lemma \label{refapprox} to approximate the measure weak* and in terms of entropy. From this it easily follows that if        $a_1,a_2\in H_p$ and $p\in (0,1)$ then $pa_1+(1-p)a_2\in H_p$.
\end{proof}

Our lower bound for the dimension of $L_a$ when $a\in H_p$ can be deduced from the following lemma which allows us to make use of the $w$-measure approach from the first part of the section.
\begin{lem} \label{lem:endpoints}
For $a\in H_p$ there exist $\{a_i\}\in H, a_i\to a$ and ergodic measures $\{\nu_i\}$ such that $\int f d\mu_i = a_i$, $\lambda(\nu_i) \to 0$ , $\lambda(\nu_i)>0$,  $\rm{supp}(\nu_i)\subset M_i$ where $M_i$ is Markov and $\dim_H \nu_i \to \dim_{\rm{hyp}}(A)$.
\end{lem}
\begin{proof}
We will let $s=\dim_{\rm{hyp}}(A)$.
Let $\varepsilon>0$ and let $\nu$ be a ergodic measure with Markov support such that $h(\nu)/\lambda(\nu)>s-\varepsilon$ and $\int f\rm{d}\nu=b$ for some $b\in H$ . We can find $\mu$ an ergodic measure with support  contained in a Markov set such that 
$$\left|\int f\rm{d}\mu-a\right|\leq\epsilon$$
and 
$$\lambda(\mu)\leq s^{-1}\epsilon^2\lambda(\mu).$$
Thus by Proposition \ref{supersyst} $(1-\varepsilon)\mu+\varepsilon\nu$ will be an invariant measure with support contained in a Markov set and with no atomic part. Moreover
$$h((1-\varepsilon)\mu+\varepsilon\nu)\geq \varepsilon h(\nu)\geq (s-2\varepsilon)\varepsilon \lambda(\nu)+s(1-\varepsilon)\lambda(\mu)$$
Thus if we let $\eta=(1-\varepsilon)\mu+\varepsilon\nu$ we have an invariant measure with support contained in a Markov set with
$$\left|\int f\rm{d}\eta-a\right|\leq \epsilon|b-a|$$
and
$$h(\eta)/\lambda(\eta)>s-2\varepsilon.$$
We can now apply Lemma \ref{invapprox}
 to find a non-atomic ergodic measure $\tilde{\nu}$ with
$$\left|\int f\rm{d}\tilde{\nu}-a\right|\leq 2\epsilon|b-a|$$
and
$$h(\tilde{\nu})/\lambda(\tilde{\nu})>s-4\varepsilon.$$
The result follows by letting $\varepsilon\to 0$.
\end{proof}

We can now obtain the lower bound in $H_p$: applying Theorem \ref{thm:w} to the sequence of measures $\{\nu_i\}$ given by Lemma \ref{lem:endpoints} we get $\dim_H L_a \geq \dim_{\rm hyp} A$. 
We finish the section by dealing with the case $\dim_H L_a\cap D_0$ when $a\in\rm{int} (H)$.

\begin{cor}
If $a\in\rm{int} (H)$ then $\delta_0(a)=\dim_{\rm{hyp}} (A)$ and thus $\dim_H L_a\cap D_0=\dim_{\rm{hyp}}(A)$.
\end{cor}
\begin{proof}
Let $\delta=\dim_{\rm{hyp}}(A)-\epsilon$ for some $\epsilon>0$. By Lemma \ref{lem:endpoints} and Proposition \ref{supersyst} we can find non-atomic ergodic measures $\mu$ and $\nu$ supported on the same Markov set   such that $\int f d\mu>a$, $\int f d\nu<a$ and $\min\{\frac{h(\mu)}{\lambda(\mu)},\frac{h(\nu)}{\lambda(\nu)}\}>\delta$. Thus for $q\geq 0$ we have that
$$\int q(f-a)+\delta\phi\rm{d}\mu>0$$
and for $q<0$ we have that
 $$\int q(f-a)+\delta\phi\rm{d}\nu>0.$$
 Thus by the variational principle (Lemma 8 \cite{H2}) $P(q(f-a)-\delta\phi)>0$ for all $q\in\R$. Since $a\in \rm{int}(H)$ we also know that $\lim_{q\to\pm\infty}P(q(f-a)-\delta\phi)=\infty$ and so $\delta_0(a)>\delta$. The result now follows.
\end{proof}
\section{Further properties of the Birkhoff spectrum}
The final step of the proof of Theorem \ref{thm:main}  to prove part 4 of the statement, that is the results on the continuity properties of the spectrum.

\begin{prop} \label{prop:upend}
The function $a\to \dim_H L_a \cap D_0$ is lower semicontinuous in $H_p$ and continuous in $H_h$.
\end{prop}
\begin{proof}
To prove the lower semicontinuity at $a$ it is enough to prove that if $a_i\to a$, $\delta_i \to \delta$, and $\inf_q P(a_i, q, \delta_i) \leq 0$ then $\inf_q P(a, q, \delta + \varepsilon) \leq 0$ for all $\varepsilon >0$. 

Assume for simplicity that $a_i\nearrow a$ (the case $a_i \searrow a$ is done analogously, and by choosing a subsequence we can always make one of those cases to happen). We can also freely assume that $a_i$ are in the interior of $H$. Fix some $\varepsilon >0$, then $\delta+\varepsilon > \delta_i$, and hence also $\inf_q P(a_i, q, \delta + \varepsilon) \leq 0$, for sufficiently large $i$.

Consider the set 

\[
\Delta_i = \{q; P(a_i, q, \delta + \varepsilon) \leq 0\}
\]

and divide it into two subsets $\Delta_i^+ = \Delta \cap (0,\infty)$ and $\Delta_i^- = \Delta \cap (-\infty,0]$. As $P(a', q, \delta) = P(a, q, \delta) + q\cdot (a-a')$, the sequence $\{\Delta_i^+\}$ is decreasing and the sequence $\{\Delta_i^-\}$ is increasing. Moreover, as $a_i$ is in the interior of $H$, $\Delta_i$ is always a compact set. Thus, $\Delta_i^-$ is compact as well.

Our assumption is that all the sets $\Delta_i$ are nonempty. If $\Delta_\ell^+$ is nonempty for some $\ell$, we can take some $q\in \Delta_\ell^-$ and this $q$ will also belong to all $\Delta_i^-$ for $i>\ell$. Hence,

\begin{equation} \label{eqn:lsc}
P(a, q, \delta+ \varepsilon) = \lim P(a_i, q, \delta+ \varepsilon) \leq 0.
\end{equation}

On the other hand, if $\Delta_i^+$ is nonempty for all $i$ then $\Delta_i^-$ must be nonempty, also for all $i$. Thus, $\{\Delta_i^-\}$ forms a decreasing family of nonempty compact sets, hence it has nonempty intersection, and for any $q\in \bigcap \Delta_i^-$ inequality \eqref{eqn:lsc} holds as well. 

Consider now $a\in H\setminus H_p$. To prove upper semicontinuity at $a$ it is enough to prove that if $a_i\to a$, $\delta_i \to \delta$, and $\inf_q P(a_i, q, \delta_i) > 0$ then $\inf_q P(a, q, \delta - \varepsilon) > 0$ for all $\varepsilon >0$.

As $a\notin H_p$, there exists $\varepsilon'>0$ such that for every measure $\mu$ satisfying $\int f d\mu \in (a-\varepsilon', a+\varepsilon')$ we have $\lambda(\mu) > \varepsilon'$. Naturally, $a_i \in (a-\varepsilon', a+\varepsilon')$ and $\delta_i > \delta-\varepsilon/2$ for $i$ large enough.

By Lemma \ref{lem:lowerint}, for all those $i$ we can find an ergodic measure $\mu_i$ with $\int f d\mu_i = a_i$ and $h(\mu_i)/\lambda(\mu_i) \geq \delta_i$. Moreover, by the above, $\lambda(\mu_i)>\varepsilon'$. Thus, for every $q\in \R$ we have 

\[
P(a_i, q, \delta - \varepsilon) \geq h(\mu_i) + \int q(f-a_i)d\mu_i - (\delta-\varepsilon)\lambda(\mu_i) = \left(\frac {h(\mu_i)} {\lambda(\mu_i)}  -\delta+\varepsilon\right) \lambda(\mu_i) \geq \frac 12 \varepsilon \varepsilon'.
\]

Thus,

\[
P(a, q, \delta - \varepsilon) = \lim_i P(a_i, q, \delta - \varepsilon) \geq \frac 12 \varepsilon \varepsilon'.
\]
\end{proof}

We know that in general $a\to \dim_H L_a\cap D_0$ will not be concave but we can present a result which shows that $a\to \dim_H L_a\cap D_0$ is piecewise monotone with at most two pieces.
\begin{prop} \label{prop:unimod}
For any $a_1 < a_2 < a_3$, 
\[
\dim_H L_{a_2} \cap D_0 \geq \min(\dim_H L_{a_1} \cap D_0, \dim_H L_{a_3} \cap D_0)
\]
\end{prop}
\begin{proof}
Fix $\delta < \min(\dim_H L_{a_1} \cap {D}_0, \dim_H L_{a_3} \cap {D}_0)$. We have for any $q\in\R$, 
$P(a_1,q,\delta)>0$ and $P(a_3,q,\delta)>0$. Moreover
$$P(a_2,q,\delta)=P(a_3,q,\delta)+q(a_2-a_3)=P(a_1,q,\delta)+q(a_2-a_1)$$
and since $a_2-a_3$ and $a_2-a_1$ have opposing signs it follows that $P(a_2,q,\delta)>0$. Thus $\dim_H L_{a_2}>\delta$ and the proof follows.

\end{proof}

\section{Examples}
One example where the Birkhoff spectra has already been studied but gives a nice illustration of the hyperbolic spectra is the Manneville-Pomeau map. Let $0<\gamma<1$ and $T:[0,1]\to [0,1]$ be defined by $T(x)=x+x^{1+\gamma}\mod 1$ if $x\in [0,1)$ and $T(1)=1$. In this case we can take $A=[0,1]$ and for $f:[0,1]\to\R$ we have that $H_p=\{f(0)\}$. 

In this setting $H=[a_{\min},a_{\max}]$ where
$$a_{\min}=\min\left\{\int f\rm{d}\mu:\mu \text{ }T-\text{inv}\right\}$$
and
$$a_{\max}=\max\left\{\int f\rm{d}\mu:\mu \text{ }T-\text{inv}\right\}.$$
We also have $H_p=\{f(0)\}$. In \cite{JJOP} it is shown that $a\to \dim_H L_{a}$ is continuous on $[a_{\min},f(0))$ and $[f(0),a_{\max}]$. In Theorem 6.3 in \cite{IJ2} it is shown it is possible to have  a discontinuity at $f(0)$ where the dimension jumps up (so the spectrum is upper-semi continuous but not lower semi-continuous). Our results  show that in this discontinuous case if we instead consider the hyperbolic spectrum $a\to \dim_H X_{a}\cap D_0$ then the spectrum is lower semi continuous and so $\dim_H L_{f(0)}\cap D_0=\liminf_{a\to f(0)}\dim_H L_{a}<\limsup_{a\to f(0)}\dim_H L_{a}=1$. (Note that it can be the case when $f(0)=a_{\min}$ that $L_{a}\cap D_0=\emptyset$. 

Another example where our results can be applied is the map $T:[0,1]\to [0,1]$ given by $T(x)=x+\gamma x^{1+\epsilon}\mod 1$ where $\gamma>1$ and $\epsilon\in (0,1)$. These maps are conjugate 
to $\beta$-expansions but have similar parabolic behaviour to the Manneville-Pomeau maps. Their thermodynamic properties (in particular uniqueness of certain equilibrium states) are studied in \cite{CT}. Our results can be applied to these maps and again $H_p=\{f(0)\}$ and the hyperbolic dimension is $1$. This means we can fully compute the spectrum both $a\to \dim_H L_{a}$ and $a\to\dim_H L_a\cap D_0$.

\newpage
\bibliography{ref}

\end{document}